\documentclass[12pt,reqno]{article}
\usepackage{fullpage}
\usepackage{amssymb}
\usepackage{amstext}
\usepackage{amsthm}
\usepackage{mathtools}
\usepackage{enumerate}  

\newtheorem{lemma}{Lemma}
\newtheorem{definition}{Definition}
\newtheorem{theorem}{Theorem}

\newtheorem{cor}{Corollary}

\newcommand{\cC}{\mathcal{C}}

\newcommand{\bC}{\mathbf{C}}

\newcommand{\malg}{\textsc{mrank}}

\DeclareMathOperator{\diag}{diag}

\DeclareMathOperator{\urank}{urank}
\DeclareMathOperator{\mrank}{mrank}
\DeclareMathOperator{\rrank}{rrank}
\DeclareMathOperator{\vdm}{VDM}
\DeclareMathOperator{\nul}{null}
\DeclareMathOperator{\rank}{rank}

\newcommand{\vertiii}[1]{{\left\vert\kern-0.25ex\left\vert\kern-0.25ex\left\vert #1 
    \right\vert\kern-0.25ex\right\vert\kern-0.25ex\right\vert}}
        
\allowdisplaybreaks

\title{Recurrence Ranks and Moment Sequences}
\author{Joshua Cooper and Grant Fickes}
\date{\today}

\begin{document}

\maketitle

\begin{abstract}
    We introduce the ``moment rank'' and ``unitary rank'' of numerical sequences, close relatives of linear-recursive order.  We show that both parameters can be characterized by a broad set of criteria involving moments of measures, types of recurrence relations, Hankel matrix factorizations, Waring rank, analytic properties of generating functions, and algebraic properties of polynomial ideals.  In the process, we solve the ``complex finite-atomic'' and ``integral finite-atomic'' moment problems: which sequences arise as the moments of a finite-atomic complex-/integer-valued measures on $\mathbb{C}$?
\end{abstract}

\section{Introduction}

\textit{This work is dedicated to the memory of Ron Graham, a gentle giant of the highest scholarly caliber, a peerless and playful teacher, and an extraordinarily generous person. His scientific contributions will live on in innumerable ways, particularly in his endless demonstrations that theory and application are not just complementary, but profoundly interwoven. Here we invoke three persistent themes of his work: recurrence relations, algorithmic thinking, and expansive TFAE statements.\\}  

We begin with a motivating problem, the original impetus for this work: Suppose that $G$ is a finite, simple graph or hypergraph; then $G$ is associated with a {\it characteristic polynomial} whose roots are its adjacency eigenvalues.  This polynomial, despite substantial attention in the literature dating back to at least \cite{CoSi57} from 1957, still holds many mysteries.  One example is the question of describing the multiplicity of zero as a root.  If $f(x)$ is the polynomial, then this multiplicity is the largest $m$ so that $g(x) = f(x)/x^m$ is also a polynomial; thus, the degree of $g(x)$ (or, actually, its reciprocal polynomial $\overline{g}(x) = x^{\deg g}g(1/x)$ in this application) then encodes this quantity as $m = \deg f - \deg g$.  Note that 
$$
\log \overline{g}(x) - \log g(0) = \sum_{i=1}^r \log(1-b_i x) = \sum_{i=1}^r \sum_{j=1}^\infty \frac{b_i^j x^j}{j}
$$
where $\{b_i\}_{i =1}^{\deg g}$ are the roots of $g$, whereupon the question becomes of bounding the smallest $r$ so that $c_j$ can be written as a sum of $r$ $j$-th powers, where $j c_j$ is the $j$-th coefficient of $\log \overline{g}$.  As will be defined below, this is exactly the ``unitary rank'' of the sequence $(jc_j)_{j \geq 1}$.  See condition (6) in Theorem \ref{Thm:TFAEforDC} for this connection with log-polynomial degree.

More generally, one might ask for the simplest recurrence that a combinatorial sequence $\mathcal{C}$ satisfies, as a kind of measure of complexity.  If $\mathcal{C}$ is ``C-finite'', then it satisfies a linear recurrence with constant coefficients, and the order of that recurrence captures this complexity.  It is natural, then, to ask how to compute this order, or even if it is finite.  Famously, for example, this is an open question for the sequence $A_n$ equal to the number of permutations of $n$ with no $1324$ pattern (i.e., $\sigma \in S_n$ so that there exist no $a<b<c<d$ so that $\sigma(a) < \sigma(c) < \sigma(b) < \sigma(d)$).  The theory of such sequences is extremely well-trodden territory, and it is simplest in the case that the characteristic polynomial -- the polynomial whose coefficients are the same as those of the recurrence -- has no repeated roots.  We focus on this case presently.

Another way in which the smallest order of a linear recurrence satisfied by a sequence appears in the literature is in the context of the venerable ``moment problem''.  Here, one asks whether a sequence can arise as the sequence of moments of various kinds of distributions: important instances include (positive) measures on $\mathbb{R}$, $[0,\infty)$, $[0,1]$, or $\mathbb{T} = \exp(i\mathbb{R})$ (the ``Hamburger'', ``Stieltjes'', ``Hausdorff'', and ``Toeplitz''/``trigonometric'' moment problems, respectively); signed measures on $\mathbb{R}$ (already considered by Hausdorff \cite{Ha23}); or atomic measures \cite{CurFiaMol2008,DioSch2017}.  See \cite{Sch2017} for an extensive exploration of this old and very broad topic.  Indeed, the question of when a sequence does {\em not} arise as moments of a finite-atomic measure was recently addressed \cite{BosElvGutMai2020}, a topic with a long history connected with totally positive matrices  \cite{FomZel2000,Pin2010}, strong log-concavity/unimodality \cite{Bre89,ChuLin15}, continued fractions and Pad\'{e} approximants \cite{ShoTam1943}.  Here, we add to the literature on moment problems by addressing the case of the underlying space being $\mathbb{C}$ with the two conditions that either (1) the measures are complex-valued and finite-atomic, or (2) the (positive) measures are finite-atomic with integer masses.

Yet another large constellation of topics closely connected with recurrence rank is the theory of Hankel matrices \cite{Lay2001,Wid1966}, matrices which are constant on anti-diagonals. These matrices -- and, more generally, Hankel operators -- play an important role in combinatorial sequence transforms \cite{Fre07,Lay2001}, numerical methods in signal processing \cite{ChuLin15}, and Riordan arrays \cite{PeaWoa00}.  The determinants of Hankel matrices, known as ``catalecticants'', are objects of study going back as far as Sylvester's work in the 1850s \cite{Syl1851}, and lives on in invariant theory \cite{Stu08}, polynomial positivity \cite{Ble13}, Waring rank and binary forms \cite{Rez2013}, and the theory of orthogonal polynomials \cite{Jun2003}.

Clearly, the subject matters connected with linear recurrence order are vast, and there is not space here to discuss them all (and many important references are therefore omitted, although they can be found by following threads in the aforementioned references).  Indeed, so much work has been done on related topics over such a long period of time that it is difficult to trace their history.  This work, in addition to presenting several new results, is an attempt to relate and unify these perspectives into one focused on the matter of linear recurrence rank, the order of the shortest recurrence a sequence satisfies, in the particularly interesting cases we term ``moment rank'' and ``unitary rank.''  We attempt to keep the below exposition mostly self-contained, which entails borrowing a variety of arguments from the literature, indicated whenever possible.

In the next section, we introduce notation, definitions, and state some basic results.  In Section 3, we present our first main theorem, a wide-ranging TFAE statement about moment rank, and discuss a few consequences.  In Section 4, we present our second main theorem, another TFAE statement about unitary rank, and some consequences thereof.

\section{Preliminaries}

Suppose the sequence $\mathcal{C} = (c_n)_{n=0}^\infty$ satisfies an $r$-th order linear recurrence relation
\begin{equation} \label{eq:recurrence}
\sum_{n=0}^{r} a_n c_{n} = 0
\end{equation}
Then, by classical results \cite{EvePooShpWar2003}, the elements of $\mathcal{C}$ can be expressed as
$$
c_n = \sum_{i=1}^r \alpha_i \beta_i^r
$$
for some $\{\alpha_i\}_{i=1}^r$, where $\{\beta_i\}_{i=1}^r$ are the roots of the degree-$r$ polynomial $p(x) = a_0 \prod_{i=1}^r (x-\beta_i) = \sum_{i=0}^{r-1} a_i x^r $, as long as the $\beta_i$ are distinct.  The $\{\alpha_i\}_{i=1}^r$ can be obtained by solving the linear system
\begin{equation} \label{eq:linsys}
\forall j \in \{0,\ldots,r-1\}, \quad \sum_{i=1}^r \alpha_i \beta_i^{j+1} = c_j
\end{equation}
These observations motivate the following definition. 

\begin{definition}\label{Def:rank}
The sequence $\mathcal{C} = (c_n)_{n=0}^N$ (with $N = \infty$ allowed) is said to have \textbf{recurrence rank} $r$ if $r$ is the smallest positive integer so that $\mathcal{C}$ satisfies a linear recurrence of order $r$. If $N = \infty$, we write $\rrank(\mathcal{C})$ for the recurrence rank.
\end{definition}

\begin{definition}\label{Def:mrank}
The sequence $\mathcal{C} = (c_n)_{n=0}^N$ (with $N = \infty$ allowed) is said to have \textbf{moment rank} $r$ if $r$ is the smallest positive integer so that there exists a set of nonzero complex numbers $\{\alpha_i\}_{i=1}^r$ and distinct nonzero $\{\beta_i\}_{i=1}^r$ so that $c_n=\sum_{i=1}^r \alpha_i \beta_i^{n+1}$ for all $0 \leq n \leq N$. If $N = \infty$, we write $\mrank(\mathcal{C})$ for the moment rank.
\end{definition}

Use of this definition depends on the uniqueness of the quantity for a given sequence. This motivates the following lemma. 

\begin{lemma} \label{lem:GDCisMeaningful}
$\mrank((c_n)_{n \geq 0})$ is well-defined.
\end{lemma}
\begin{proof}
Suppose, by way of contradiction, that there are two sets $\{\beta_i\}_{i=1}^r$ and $\{\beta'_j\}_{j=1}^s$ of nonzero distinct complex numbers along with sets of nonzero complex numbers $\{\alpha_i\}_{i=1}^r$ and $\{\alpha'_j\}_{j=1}^s$ so that  
$$
c_n=\sum_{i=1}^r\alpha_i \beta_i^{n+1}=\sum_{j=1}^s\alpha'_j \beta_j^{\prime(n+1)}
$$ 
Write $f(z)=\sum_{n=0}^\infty c_n z^n$. Because $\max_i|\beta_i|$ and $\max_j|\beta'_j|$ are finite, the following is true around a sufficiently small ball about $z=0$: 
\begin{align*} 
\sum_{n=0}^\infty \sum_{i=1}^r z^n\alpha_i \beta_i^{n+1} &= \sum_{n=0}^\infty \sum_{j=1}^s z^n \alpha'_j \beta_j^{\prime (n+1)} \\
\sum_{i=1}^r \sum_{n=0}^\infty z^n\alpha_i \beta_i^{n+1} &= \sum_{j=1}^s \sum_{n=0}^\infty z^n \alpha'_j \beta_j^{\prime (n+1)} \\
\sum_{i=1}^r \frac{\alpha_i \beta_i}{1-z\beta_i} &= \sum_{j=1}^s \frac{\alpha'_j \beta'_j}{1-z\beta'_j}
\end{align*}

These two functions are equal, so they have the same set of (simple) poles; thus, $\{\beta_i\}_i = \{\beta'_j\}_j$ and $r=s$.  Furthermore, since the residues of these poles are proportional to the multiplicity of the $\alpha_i \beta_i$ and $\alpha'_j \beta'_j$ values, we also have that $\{\alpha_i\}_{i =1}^r = \{\alpha'_j\}_{j =1}^s$.
\end{proof}

The beginning of this section motivates the moment rank definition by the satisfaction of a linear recurrence. The following definition introduces a specific kind of linear recurrence we consider throughout the work.

\begin{definition}
An $r$-th order linear recurrence of the form $\sum_{j=0}^r a_j c_{j+t} = 0$ satisfied by the sequence $(c_j)_{j \geq 0}$ for all $t\geq 0$ is said to be {\em simple} if the characteristic polynomial $\sum_{j=0}^r a_j x^j$ has distinct roots. 
\end{definition}

The characteristic polynomial of a recurrence provides a way to translate between polynomials and recurrences. Given a sequence which satisfies a linear recurrence, there are methods to define other recurrences of higher orders which the given sequence satisfies. We consider this idea in the context of characteristic polynomials, motivating the following definition and the lemma that follows. 

\begin{definition}
Given a complex sequence $\mathcal{C} = (c_n)_{n=0}^\infty$, let 
$$ 
R_{\mathcal{C}}= \left \{\sum_{j=0}^r a_j x^j : a_j\in\mathbb{C}\mbox{ and }\sum_{j=0}^r a_j c_{j+t} = 0 \mbox{ for all }t\geq 0 \right \}
$$
be the set of characteristic polynomials of arbitrary finite order linear recurrences the sequence $\mathcal{C}$ satisfies for all $t\geq 0$. We note that the zero polynomial is a trivial element of $R_{\mathcal{C}}$. 
\end{definition}

The set above is a subset of one variable polynomials with complex coefficients. We explore useful algebraic properties of this subset of $\mathbb{C}[x]$. 

\begin{lemma}
Given a complex sequence $\mathcal{C} = (c_n)_{n=0}^\infty$, the set $R_{\mathcal{C}}$ is an ideal in $\mathbb{C}[x]$. 
\end{lemma}
\begin{proof}
We know the zero polynomial is in $R_{\mathcal{C}}$ by definition. Let $a(x) = \sum_{i=0}^r a_i x^i$ and $b(x) = \sum_{j=0}^s b_j x^j$ be arbitrary elements of $R_{\mathcal{C}}$. Then $\sum_{i=0}^r a_i c_{i+t} = 0$ and $\sum_{j=0}^s b_j c_{j+t} = 0$, giving that 
$$
\sum_{i=0}^r a_i c_{i+t} + \sum_{j=0}^s b_j c_{j+t} = 0. 
$$
Thus, $R_{\mathcal{C}}$ is closed under addition.

Now let $\{d_k\}_{k=0}^q$ be complex constants so that $d(x)\in R_{\mathcal{C}}$ where $d(x) = \sum_{k=0}^q d_k x^k$. Let $p(x) = \sum_{l=0}^m p_l x^l$ be an arbitrary polynomial with complex coefficients. If $d(x)$ or $p(x)$ is the zero polynomial, then $p(x)\cdot d(x) \in R_{\mathcal{C}}$. Suppose now that $d(x)$ and $p(x)$ are not identically zero. Since the sequence $\mathcal{C}$ satisfies a recurrence with characteristic polynomial $d(x)$, the generating function $\Phi(x)$ of the sequence has denominator $d(x)$. Therefore $p(x)\cdot d(x)\cdot \Phi(x)$ is a polynomial, so $p(x)\cdot d(x)\cdot \Phi(x)$ is the characteristic polynomial for a recurrence satisfied by $\mathcal{C}$, giving that $R_{\mathcal{C}}$ is closed under multiplication by elements of $\mathbb{C}[x]$.
\end{proof}

Note that since $\mathbb{C}[x]$ is a principal ideal domain, $R_{\mathcal{C}}$ is generated by one complex polynomial. Moreover, we call $R_{\mathcal{C}}$ the {\em recurrence ideal} of the sequence $\mathcal{C}$. 

\begin{cor}\label{Cor:SimpleGenerator}
Given a complex sequence $\mathcal{C} = (c_n)_{n=0}^\infty$ let $R_{\mathcal{C}}$ be generated by $p(x)$. If $p(x)$ has repeated roots, then $\mathcal{C}$ does not satisfy a simple linear recurrence of any order. Thus, $\mathcal{C}$ satisfying a simple linear recurrence implies that $p(x)$ has distinct roots, i.e., if $\mrank(\mathcal{C}) <\infty$, then $\rank(\mathcal{C}) = \mrank(\mathcal{C})$. 
\end{cor}

The algebraic structure of the recurrence ideal gives rise to useful properties of simple linear recurrences, some of which are investigated by Lemma \ref{Lem:SimpleHelp}. The properties addressed in the following two lemmas are useful in the proof of Theorem \ref{Thm:TFAEforGDC}.

\begin{lemma}\label{Lem:SimpleHelp}
Let $r$ be the smallest positive integer so that the sequence $(c_j)_{j \geq 0}$ satisfies the simple $r$-th order linear recurrence $\sum_{j=0}^r a_j c_{j+t} = 0$. Then the following observations hold. 
\begin{enumerate}
\item $a_r\neq 0$. 
\item The roots of $p(x) = \sum_{j=0}^r a_j x^{j}$ are non-zero.
\item The polynomial $q(x) = \sum_{j=0}^r a_j x^{r-j}$ (the ``reciprocal" of the characteristic polynomial) has $r$ distinct, nonzero roots. 
\end{enumerate}
\end{lemma}
\begin{proof}
(1) - This follows trivially from the minimality of $r$. 

(2) - Due to Corollary \ref{Cor:SimpleGenerator}, $p(x)$ is the generator of the recurrence ideal, $R_{\mathcal{C}}$. If $x$ is a factor of $p(x)$, that corresponds to an index shift in the corresponding recurrence. Then $a(x) = p(x)/x$ is also the characteristic polynomial of a linear recurrence satisfied by $\mathcal{C}$, so $a(x)\in R_{\mathcal{C}}$, contradicting the minimality of $\deg(p)$.

(3) - Since $a_r\neq 0$ by (1), $q(x)$ has nonzero constant term so $q(0) \neq 0$. Therefore, the reciprocal of $q(x)$ is well defined, with $p(x)$ the reciprocal of $q(x)$. The polynomial $p(x)$ has $r$ distinct nonzero roots by (2) and simplicity, so the same can be said of $q$, whose roots are the reciprocals of the roots of $p$. 
\end{proof}

\begin{lemma}\label{Lem:UniqueRecurrence}
Let the sequence $\cC = (c_n)_{n \geq 0}$ satisfy two $r$-th order linear recurrences, namely the minimal order recurrence $\sum_{n=0}^r a_n c_{n+t} = 0$ and another $r$-th order recurrence $\sum_{n=0}^r b_n c_{n+t} = 0$, for all $t\geq 0$. Then $(a_1,a_2,\dotsc,a_r) = \lambda (b_1,b_2,\dotsc,b_r)$, where $\lambda \neq 0$ is a scalar.
\end{lemma}
\begin{proof}
Let $a(x)$ be the characteristic polynomial of the recurrence $\sum_{n=0}^r a_n c_{n+t} = 0$. Since $r$ is minimal, we have that $a(x)$ generates $R_{\cC}$. If $b(x)$ is the characteristic polynomial of the recurrence $\sum_{n=0}^r b_n c_{n+t} = 0$, it must be that $a(x) = \lambda b(x)$ for some $\lambda\in\mathbb{C} \setminus \{0\}$ since $R_{\cC}$ is principal and $\deg(a)=\deg(b)$, from which the result follows. 
\end{proof}

We have already seen that the roots of characteristic polynomials associated with simple linear recurrences are distinct. 
The discriminant is a polynomial in the coefficients of univariate complex polynomials, whose kernel is exactly the set of polynomials with a repeated root.  This kernel is the ``discriminant variety''.

\begin{definition}
Fix the natural number $r\geq 1$. Then the (affine) \textbf{$r$-discriminant variety}, denoted $\nabla_r$, is the closure of 
$$
\left \{(b_0,\dotsc,b_{r}) \in \mathbb{C}^{r+1} : f(x) = \sum_{i=0}^r b_i x^i \mbox{ has a repeated root} \right \}
$$
\end{definition}

It is also common to consider a Hankel matrix whose entries are given by the elements of a sequence. Both finite and infinite dimensional square Hankel matrices are considered throughout the paper. In \cite{BolLukVan}, the authors show that all infinite Hankel matrices have generalized Vandermonde decompositions of a specified form, dependent on the recurrences the original sequence satisfies. Our investigation into simple linear recurrences invites the question of which additional matrix properties are satisfied by Hankel matrices generated by sequences satisfying simple linear recurrences. The following lemma and subsequent definition provide tools necessary to analyze the structure of these matrices. 

\begin{lemma}\label{Lem:VandeComp}
Let $V$ be an $r\times n$ Vandermonde matrix where the $(i,j)$ entry is $a_i^{j-1}$, and let $D$ be an $r\times r$ diagonal matrix with $(i,i)$ entry $b_i$. Take $a_i$ and $b_i$ for $1\leq i\leq r$ to be complex scalars. Then the matrix $V^TDV$ is a Hankel matrix. 
\end{lemma}
\begin{proof}
We simply perform the matrix multiplication, showing the form of each product along the way. Let the matrices $D_{r\times r}$ and $V_{n\times r}$ be defined as follows, where $\{a_i\}_{i=1}^r$ and $\{b_i\}_{i=1}^r$ are complex scalars. Let $D_{r\times r} = \diag(b_i)$ and 
$$
V = \begin{pmatrix}
1 & a_1 & a_1^2 & \cdots & a_1^{n-1} \\
1 & a_2 & a_2^2 & \cdots & a_2^{n-1} \\
\vdots & \vdots & \vdots & \ddots & \vdots \\
1 & a_{r} & a_{r}^2 & \cdots  & a_r^{n-1}
\end{pmatrix}_.
$$
We see that  
$$
V^TD = \begin{pmatrix}
b_1 & b_2 & b_3 & \cdots & b_r \\
b_1 a_1 & b_2 a_2 & b_3 a_3 & \cdots & b_r a_r \\
\vdots & \vdots & \vdots & \ddots & \vdots \\
b_1 a_1^{n-1} & b_2 a_2^{n-1} & b_3 a_3^{n-1} & \cdots  & b_r a_r^{n-1}
\end{pmatrix}_,
$$
and furthermore that 
$$
V^TDV = \begin{pmatrix}
\sum_{i=1}^r b_i & \sum_{i=1}^r b_i a_i & \sum_{i=1}^r b_i a_i^2 & \cdots & \sum_{i=1}^r b_i a_i^{n-1} \\
\sum_{i=1}^r b_i a_i & \sum_{i=1}^r b_i a_i^2 & \sum_{i=1}^r b_i a_i^3 & \cdots & \sum_{i=1}^r b_i a_i^{n} \\
\vdots & \vdots & \vdots & \ddots & \vdots \\
\sum_{i=1}^r b_i a_i^{n-1} & \sum_{i=1}^r b_i a_i^n & \sum_{i=1}^r b_i a_i^{n+1} & \cdots & \sum_{i=1}^r b_i a_i^{2n-2} \\
\end{pmatrix}_.
$$
Then $c_j = \sum_{i=1}^r b_i a_i^j$ for $0\leq j\leq 2n-2$ is the sequence which populates the Hankel matrix $V^T D V$. 
\end{proof}

We are specifically interested in Vandermonde matrices with no zero entries, motivating the following definition. 

\begin{definition}\label{Def:no0Vande}
Let $H_{n\times n}$, where $n$ is allowed to be $\infty$, be a complex matrix. We say $H$ has a \textbf{non-degenerate Vandermonde decomposition} if there exists a Vandermonde matrix $V_{r\times n}$ with all entries nonzero and a diagonal matrix $D_{r\times r}$ so that $H_\infty = V^TDV$. 
\end{definition}

\section{Moment Rank}

Finally, before presenting the main theorem of this section, we describe an algorithm which returns the moment rank of a sequence and the coefficients of a linear combination of powers witnessing to this rank. Denote the $r \times r$ (modified) Vandermonde matrix with $(i,j)$ entry $\beta_i^{j}$ for $\boldsymbol{\beta} = (\beta_1,\ldots,\beta_r)$ by $\vdm'(\boldsymbol{\beta})$, and the $r \times r$ (ordinary) Vandermonde matrix with $(i,j)$ entry $\beta_i^{j-1}$ by $\vdm(\boldsymbol{\beta})$.\\

\noindent \textbf{Algorithm} $\malg(\mathcal{C})$: Given the sequence $\mathcal{C} = (c_n)_{n\geq 0}$, set $r=0$.  Then
\begin{enumerate}
    \item $r \leftarrow r+1$
    \item Let $\mathbf{c}_t = (c_t,\ldots,c_{t+r-1})$, and $\mathbf{C} = (\mathbf{c}_0^T, \ldots, \mathbf{c}_{r-1}^T)$.
    \item If $\det(\mathbf{C}) = 0$, \texttt{goto} step 1.  Else \texttt{continue}.
    \item Let $(a_0,\ldots,a_{r-1})^T = -\mathbf{C}^{-1} \mathbf{c}_r^T$, and define $p(x) = \sum_{n=0}^{r} a_n x^{n}$, where $a_r = 1$.
    \item If $(c_n)_{n \geq 0}$ does not satisfy the recurrence $\sum_{i=0}^{r} a_i c_{i+t} = 0$ for all $t \geq 0$, \texttt{goto} step 1.  Else \texttt{continue}.
    \item If $p$ has repeated roots, \texttt{throw} \texttt{ErrorNotSimple} and \texttt{terminate}.  Else \texttt{continue}.
    \item \texttt{return} $r$ and $\boldsymbol{\alpha} = (\mathbf{c}_0 \vdm'(\boldsymbol{\beta})^{-1})^T$, where $\boldsymbol{\beta}$ is the vector of roots of $p(x)$.
\end{enumerate}

Note that the above is in truth only a {\em template} for an actual executable algorithm, since some steps involve unspecified subroutines, such as computation of the determinant in Step 3, or checking for repeated roots in Step 6. Indeed, Step 5 involves checking whether a sequence is satisfied by a given recurrence, a task which could range from very straightforward (e.g., if the sequence was given as the solution to a linear recurrence) to undecidable (e.g., if the sequence is not a computable function of its index).  We therefore make no attempt to analyze the complexity of $\malg(\cdot)$ and instead treat constitutive subproblems as black boxes.  However, we do assume that each step is indeed computable in the sense that there exists an algorithm which will, in finite time, return \texttt{True} or \texttt{False} correctly.\\

The following is our main theorem concerning sequences of finite moment rank.  We show that all of the above contexts provide interpretations of the moment rank.  Recall that an $S$-measure on a space $(\mathcal{X},\Sigma)$, where $\Sigma$ is a $\sigma$-algebra on $\mathcal{X}$, is a countably additive function from $\Sigma$ to $S$, where $S$ is an additive monoid with limits such as $[0,\infty)$ (positive measure), $\mathbb{R}$ (signed measure), or $\mathbb{C}$ (complex measure). The $t$-th moment $m_t$ of an $S$-measure $d\mu$ on $\mathcal{X}$ is the quantity $\int_\mathcal{X} x^t \, d\mu$, and the sequence $\{m_t\}_{t =0}^\infty$ is its ``moment sequence''.  We call a measure ``$r$-atomic'' if there exists an $A \subset \mathcal{X}$ with $|A|=r$ so that $\mu(x) \neq 0$ for $x \in A$, and $\mu(B) = 0$ for any $B \subset \mathcal{X} \setminus A$ in $\Sigma$.

\begin{theorem}\label{Thm:TFAEforGDC}
Suppose $\mathcal{C} = (c_n)_{n=0}^\infty$ is a sequence in $\mathbb{C}$, and $r \in \mathbb{N}$. Let $H_{m,t}$ denote the $(m+1)\times(m+1)$ Hankel matrix whose entries come from the sequence $(c_n)_{n=t}^{2m+t}$, and let $f = \sum_{n \geq 0} c_n z^n$ denote the ordinary generating function of $\mathcal{C}$. Then the following are equivalent. 
\begin{enumerate}
\item The sequence $\mathcal{C}$ has moment rank $r$. 
\item $\mathcal{C}$ satisfies a simple $r$-th order linear recurrence, and $r$ is the smallest positive integer so that $\mathcal{C}$ has this property.
\item The matrices $H_{m,t}$ satisfy $\det(H_{r-1,t})\neq 0$, $\nul(H_{r,0}) = 1$, $\ker(H_{r,t}) = \ker(H_{r,0})$ for every $t \geq 0$, and $\ker(H_{r,0})\not\subseteq \nabla_r$.
\item There exist $\{\alpha_1,\ldots,\alpha_r\}, \{\beta_1,\ldots,\beta_r\}, \{\lambda_1,\ldots,\lambda_r\} \subset \mathbb{C} \setminus \{0\}$ so that, for each $t \geq 0$, the polynomial $\sum_{j=0}^{2r} \binom{2r}{j} c_{j+t} x^{2r-j} y^{j} = \sum_{j=1}^r \lambda_j (\beta_j/\alpha_j)^t (\alpha_j x + \beta_j y)^{2r}$ and $\{\alpha_j/\beta_j\}_{j=1}^r$ is a set of $r$ distinct values.
\item The ordinary generating function $\Phi(z) = \sum_{n \geq 0} c_n z^n$ of $\mathcal{C}$ is a rational function with exactly $r$ simple poles.
\item The infinite Hankel matrix $H_\infty$ has rank $r$ and admits a non-degenerate Vandermonde decomposition. 
\item The ideal $R_{\mathcal{C}}$ is radical, and the least degree of any nonzero element is $r$.
\item The sequence $\cC$ is the moment sequence for a complex $r$-atomic measure on $\mathbb{C}$. 
\item The algorithm $\malg(\mathcal{C})$ returns the parameter $r$.
\end{enumerate}
\end{theorem}
\begin{proof}
\hfill
\begin{description}
\item[$2 \Leftrightarrow 7$:] Suppose $7$. Let $p(x)$ be a monic generator of $R_{\mathcal{C}}$, which exists because $\mathbb{C}[x]$ is a PID. We know that $p$ is the characteristic polynomial for a linear recurrence of minimal order satisfied by $\mathcal{C}$. Note that $\deg(p) = r$. Let $\{\beta_i\}_{i=1}^s$ be the distinct roots of $p$, and let $m$ be the maximum multiplicity of a root of $p$. Consider the polynomial $q(x) = \prod_{i=1}^s (x - \beta_i)^m$. Clearly $p$ divides $q$, giving that $q(x)\in R_{\mathcal{C}}$. Since the ideal is radical, we have that $\sqrt[m]{q(x)} = \prod_{i=1}^s (x-\beta_i)$ is an element of $R_{\mathcal{C}}$. Thus, $\deg(\sqrt[m]{q(x)}) \leq \deg(p(x))$, and $p(x)$ generating the recurrence ideal implies $\deg(\sqrt[m]{q(x)}) = \deg(p(x))$, so $p(x)$ has distinct roots.

\item We prove the reverse direction by contraposition. Let $p(x)$ be the generator of $ R_{\mathcal{C}}$ and suppose $R_{\mathcal{C}} = \langle p(x) \rangle$ is not radical. Let $f(x)\in R_{\mathcal{C}}$ and $m \geq 2$ be an integer so that $\sqrt[m]{f(x)}$ is an element of $\mathbb{C}[x]\setminus R_{\mathcal{C}}$. Since $R_{\mathcal{C}}$ is principal, we have that $p$ divides $f$, but $p$ does not divide $\sqrt[m]{f}$. Since the distinct roots of $f$ and $\sqrt[m]{f}$ are the same, we have that $p$ does not have distinct roots. Since $p$ does not have distinct roots, the same can be said of every polynomial in $R_{\mathcal{C}}$ and so $\mathcal{C}$ does not satisfy a simple linear recurrence.

\item[$2 \Rightarrow 5 \Rightarrow 4$:] We adapt an argument from \cite{Rez2013}, ultimately drawing upon Sylvester's legendary manuscript \cite{Syl1851}.  Suppose $(2)$, giving that for $\mathcal{C}$ we have $\sum_{n=0}^r a_n c_{n+t} = 0$ for $t\geq0$, where $r$ is the smallest order simple recurrence the sequence satisfies. By Lemma (\ref{Lem:SimpleHelp}), the polynomial $g(x) = \sum_{n=0}^r a_n x^{r-n}$ has no repeated roots, and $a_r\neq 0$. Define $h(x,y) = \sum_{n=0}^r a_nx^{r-n}y^{n}$. Without loss of generality let $a_r = 1$. Let $\alpha_n$ and $\beta_n$ be complex numbers for $1\leq n\leq r$ so that $h(x,y) = \prod_{n=1}^r (-\beta_n x+\alpha_n y)$. Note that the $\frac{\alpha_n}{\beta_n}$ are distinct since $h(x,1) = g(x)$ has distinct roots. 

Let $\Phi(T) = \sum_{m=0}^\infty c_m T^m$. Then we have the following, which converges within a positive-radius disk about zero:
$$ \left(\sum_{n=0}^r a_{r-n}T^n\right)\Phi(T) = \sum_{j=0}^{r-1}\sum_{k=0}^j a_{r-(j-k)}c_k T^j + \sum_{j=r}^\infty \sum_{k=0}^r a_{r-k}c_{j-k}T^j $$
In the second term above, we have $j-r\geq 0$. Therefore, $\sum_{k=0}^r a_{r-k}c_{j-k} = \sum_{n=0}^r a_n c_{n+(j-r)}$ and the second term vanishes, leaving 
$$ \left(\sum_{n=0}^r a_{r-n}T^n\right)\Phi(T) = \sum_{j=0}^{r-1}\sum_{k=0}^j a_{r-(j-k)}c_k T^j $$
Thus, $\Phi(T)$ is a rational function with denominator $\sum_{n=0}^r a_{r-n}T^n = \sum_{n=0}^r a_{n}T^{r-n} = h(T,1) = \prod_{n=1}^r (\alpha_n-\beta_nT )$. Since the $\frac{\alpha_n}{\beta_n}$ are distinct, this completes the proof of $(5)$. 

Continuing from here, by partial fractions, since the $\frac{\alpha_n}{\beta_n}$ are distinct there exist $\lambda_n$ for $1\leq n\leq r$ so that (choosing numerators with foresight)
$$ \Phi(T) = \sum_{n=1}^r \frac{\lambda_n\alpha_n^{2r+1}}{\alpha_n-\beta_nT} \Rightarrow c_m = \sum_{n=1}^r \lambda_n\alpha_n^{2r}\left(\frac{\beta_n}{\alpha_n}\right)^m.$$ 
The following computation completes the proof of (4), noting that the minimality of $r$ in this setting is due to the construction in $4 \Rightarrow 2$ (see below). 
\begin{align*}
\sum_{j=0}^{2r} \binom{2r}{j} c_{j+t} x^{2r-j} y^{j} &= \sum_{n=1}^r \lambda_n\alpha_n^{2r} \left(\frac{\beta_n}{\alpha_n}\right)^{t} \sum_{j=0}^{2r} \binom{2r}{j}  \left(\frac{\beta_n}{\alpha_n}\right)^{j} x^{2r-j} y^{j} \\
&= \sum_{n=1}^r \lambda_n (\beta_n/\alpha_n)^t (\alpha_n x + \beta_n y)^{2r}
\end{align*}

\item[$4 \Rightarrow 2$:] Suppose $(4)$, giving that there exist nonzero $\{\alpha_1,\ldots,\alpha_r\}$, $\{\beta_1,\ldots,\beta_r\}$, and $\{\lambda_1,\ldots,\lambda_r\}$ so that, for each $t \geq 0$, the polynomial $\sum_{j=0}^{2r} \binom{2r}{j} c_{j+t} x^{2r-j} y^{j} = \sum_{j=1}^r \lambda_j (\beta_j/\alpha_j)^t (\alpha_j x + \beta_j y)^{2r}$, the set $\{\alpha_j/\beta_j\}_{j=1}^r$ consists of $r$ distinct values, and $r$ is the smallest positive integer for which this property holds. Then for $0\leq j \leq 2r$ we have $$ c_{j+t} = \sum_{j=1}^r \lambda_j (\beta_j/\alpha_j)^t (\alpha_j^{2r-j}\beta_j^{j}) = \sum_{j=1}^r \lambda_j \alpha_j^{2r-j-t} \beta_j^{j+t}.$$ 

Let $h(x,y) = \prod_{n=1}^r (-\beta_n x+\alpha_n y)$. Moreover, let $a_n$ for $0\leq n \leq r$ so that $h(x,y) = \sum_{n=0}^r a_nx^{r-n}y^{n}$. Note that $a_r = \prod_{n=1}^r \alpha_n$. Since $\alpha_n\neq0$ for all $1\leq n\leq r$, we have that $a_r\neq 0$. Continuing, we have 
\begin{align*}
\sum_{n=0}^r a_n c_{n+t} &= \sum_{j=1}^r \sum_{n=0}^r a_n \lambda_j \alpha_j^{2r-n-t} \beta_j^{n+t} \\
& = \sum_{j=1}^r \lambda_j \alpha_j^{r-t} \beta_j^{t} \sum_{n=0}^r a_n  \alpha_j^{r-n} \beta_j^{n} \\
&= \sum_{j=1}^r \lambda_j \alpha_j^{r-t} \beta_j^{t} h(\alpha_j,\beta_j) = 0. 
\end{align*}
Therefore, the sequence satisfies an $r$-th order linear recurrence. Notice that the recurrence is simple since $h(x,1) = \prod_{n=1}^r (-\beta_n x+\alpha_n )$ has distinct, nonzero roots. Moreover, the minimality of $r$ in this setting is due to the construction in the preceding argument. 
\item[$5 \Rightarrow 7$:] Suppose $(5)$. Let $\Phi(z) = \sum_{n=0}^\infty c_n z^n$. Let the $r$ simple poles of $\Phi(z)$ be $\gamma_n$ for $1\leq n\leq r$ so that $g(z) = \prod_{n=1}^r (z-\gamma_n)$ is the denominator of $\Phi(z)$. It is a well-known result \cite{EvePooShpWar2003} that the polynomial $f(z) = \sum_{i=0}^m a_i z^i$ is the characteristic polynomial of a recurrence satisfied by $\mathcal{C}$ if and only if $f(z)\Phi(z)$ is a polynomial. Thus $f(z)\in R_{\mathcal{C}}$ if and only if $g(z)$ divides $f(z)$, giving that $g(z)$ (a polynomial of degree $r$) is a generator for $R_{\mathcal{C}}$. The distinctness of the roots of $g(z)$ implies $R_{\mathcal{C}}$ is radical. 

\item[$1 \Leftrightarrow 2$:] Suppose (1), meaning there exists a set of nonzero complex numbers $\{\alpha_i\}_{i=1}^r$ and distinct nonzero $\{\beta_i\}_{i=1}^r$ so that $c_n=\sum_{i=1}^r \alpha_i \beta_i^{n+1}$ for all $n\geq 0$. We multiply both sides of the equation by $z^n$ and take the sum over $n$ to examine the ordinary generating functions of both sides. We have the following. 
$$ \Phi(z) = \sum_{n=0}^\infty \sum_{i=1}^r \alpha_i \beta_i^{n+1} z^n = \sum_{i=1}^r \sum_{n=0}^\infty \alpha_i \beta_i^{n+1} z^n  = \sum_{i=1}^r  \frac{\alpha_i \beta_i}{1 - \beta_i z} 
$$ 
Thus, if we let $g(z) = \prod_{i=1}^r (1-\beta_i z) = \sum_{i=0}^r a_i z^i$ for some coefficients $a_i$, then $\Phi(z)g(z) = \sum_{i=0}^{r-1} \lambda_i z^i$ is a polynomial of degree $r-1$ where $\lambda_{r-1} \neq 0$, i.e.,
$$
\sum_{i=0}^{r-1} \lambda_i z^i = \sum_{n=0}^\infty c_n z^n g(z) = \sum_{n=0}^\infty c_n z^n \sum_{i=0}^r a_i z^i = \sum_{n=r}^\infty z^n \sum_{i=0}^r a_i c_{n-i} + \sum_{n=0}^{r-1} z^n \sum_{i=0}^n a_i c_{n-i}.
$$
By matching coefficients of $z^n$ on each side, we see, for $n \geq r$,
$$
0 = \sum_{i=0}^r a_i c_{n-i}.
$$
But, $a_0 = 1$, so
\begin{equation} \label{eq:actualrecurrence}
c_n = - \sum_{i=1}^r a_i c_{n-i},
\end{equation}
i.e., if $\boldsymbol{\alpha} = (a_r,a_{r-1},\ldots,a_1)$, we have $\bC \boldsymbol{\alpha}^T = - \mathbf{c}_r^T$,
where $\mathbf{c}_t := (c_{t}, \cdots , c_{t+r-1})$ and $\mathbf{C} = (\mathbf{c}_{0}, \cdots , \mathbf{c}_{r-1})^T$. So $\boldsymbol{\alpha}^T = - \bC^{-1} \mathbf{c}_r^T$, and

$q(x) = \sum_{i=0}^{r} a_i x^{r-i}$ is the characteristic polynomial of the recurrence (\ref{eq:actualrecurrence}).  Since the recurrence is solved by $c_n = \sum_{i=1}^r \alpha_i \beta_i^{n+1}$, the roots of $p(x)$ are the distinct nonzero values $\{\beta_i\}_{i=1}^r$. Moreover, since each $\beta_i\neq 0$, we have that $a_r\neq 0$. Thus, $\mathcal{C}$ satisfies a simple $r$-th order linear recurrence, and the minimality of $r$ in this setting is given by the construction in the other direction of the proof (below) together with Lemma \ref{lem:GDCisMeaningful}. 
\item Suppose $(2)$, giving that 
$\sum_{n=0}^r a_n c_{n+t} = 0$ for all $t\geq0$, and $r$ is the smallest order recurrence the sequence satisfies. 
By Lemma \ref{Lem:SimpleHelp}, let $\{\beta_i\}_{i=1}^r$ be the distinct nonzero roots of $p(x)$, the characteristic polynomial of the recurrence. It follows from standard facts about rational functions that $p(x)$ having distinct roots implies $c_n = \sum_{i=1}^r \alpha_i \beta_i^{n+1}$ for all $n\geq 0$ where each $\alpha_i \in \mathbb{C}$. If any of the $\alpha_i$ are zero, the construction in the other direction of the proof would contradict the minimality of $r$ in this direction. Therefore all $\alpha_i$ are nonzero and $\mathcal{C}$ has moment rank $r$.

\item[$3 \Leftrightarrow 2$:] Suppose $(2)$, giving that 
$\sum_{n=0}^r a_n c_{n+t} = 0$ for all $t\geq0$ with $a_r = 1$, and $r$ is the smallest order recurrence the sequence satisfies. The $r$-th order linear recurrence implies $H_{r,t}A = 0$, where $A = [a_0,a_1,\dotsc,a_r]^T$. Let $M$ be the $(r+1) \times (r+1)$ matrix given by $M_{i,j} = 1$ if $j = i+1$, $M_{r+1,j} = -a_{j-2}-a_{j-1}$ for $1 \leq j \leq r+1$ where we define $a_{-1} = 0$, and $M_{i,j} = 0$ otherwise.  Then $M$ is invertible, because $M_{r+1,1} \neq 0$, the $(r+1,1)$-minor of $M$ is $1$, and $M (c_t,\ldots,c_{r+t})^T = (c_{t+1},\ldots,c_{r+t+1})^T$ because $M$ acts as a left-shift on the first $r$ coordinates, and the $r+1$-st coordinate is given by
\begin{align*}
\sum_{j=1}^{r+1} -(a_{j-1}+a_{j-2}) c_{j+t-1} & = \sum_{j=1}^{r+1} -a_{j-2} c_{j+t-1} - \sum_{j=1}^{r+1} a_{j-1} c_{j+t-1} \\
& = a_{r} c_{r+t+1} - \sum_{n=0}^{r} a_{n} c_{n+t+1} - \sum_{n=0}^{r} a_{n} c_{n+t} \\
& = c_{r+t+1}.
\end{align*}
Therefore, $H_{r,t} = M^k H_{r,t-k}$ for any integers $t \geq \max\{0,k\}$, so $\ker(H_{r,t}) = \ker(H_{r,t'})$ for every $t,t' \geq 0$. Suppose that there exists a non-trivial vector $B = [b_0,b_1,\dotsc,b_r]^T$ contained within $\ker(H_{r,t})$. Then $\sum_{n=0}^r b_n c_{n+t} = 0$ for all $t\geq0$, and Lemma \ref{Lem:UniqueRecurrence}, gives that $B$ is a scalar multiple of $A$. Thus, $\nul(H_{r,t})= 1$.  The polynomial $p(x) = \sum_{i=0}^r a_i x^{r}$ has no repeated roots, implying that $\ker(H_{r,t}) \not\subseteq \nabla_r$.

Finally, the minimality of $r$ gives that $\nul(H_{m,t}) = 0$ for all $1\leq m\leq r-1$, implying that $\det(H_{r-1,t}) \neq 0$ for all $t\geq 0$. 

\item Suppose $(3)$. There exists $A = [a_0,a_1,\dotsc,a_r]^T$ so that $A\in\ker(H_{r,t})$ for all $t\geq0$. Moreover, $(a_0,\dotsc,a_r) \notin \nabla_r$, so the polynomial $a(x) = \sum_{i=0}^r a_i x^i$ has distinct complex roots. Let the first row of $H_{r,t} = \textbf{c}_t = [c_{t},\dotsc,c_{t+r}]$. Then $\textbf{c}_t \cdot A = 0$, giving that $0 = \sum_{i=0}^r a_i c_{i+t}$ for all $t\geq 0$, and the original sequence satisfies an $r$th order linear recurrence. The minimality of $r$ is given by $\det(H_{r-1,t})\neq 0$ for all $t\geq 0$, completing the proof. 

\item[$1 \Leftrightarrow 6:$] Suppose (6). Let $H_\infty$ (with entries from the sequence $\mathcal{C}$) have rank $r$ and admit a non-degenerate Vandermonde decomposition. Let $\{a_i\}_{i=1}^r$ and $\{b_i\}_{i=1}^r$ be sets of complex scalars so that $D_{r\times r} = \diag(b_i)$, 
$$
V_{r\times\infty} = \begin{pmatrix}
1 & a_1 & a_1^2 & \cdots & a_1^{n-1} & \cdots \\
1 & a_2 & a_2^2 & \cdots & a_2^{n-1} & \cdots \\
\vdots & \vdots & \vdots & \ddots & \vdots & \cdots \\
1 & a_{r} & a_{r}^2 & \cdots  & a_r^{n-1} & \cdots
\end{pmatrix}_,
$$
and $H_\infty = V^TDV$. Since $\rank(H_\infty) = r$, $\rank(V) \leq r$, and $\rank(D) \leq r$, we have that $\rank(V) = \rank(D) = r$ and furthermore that $b_i\neq 0$ for $1\leq i\leq r$. By the computation given in Lemma \ref{Lem:VandeComp}, we see that 
\begin{align*}
\begin{pmatrix}
c_0 & c_1 & \cdots & c_{n-1} & \cdots \\
c_1 & c_2 & \cdots & c_{n} & \cdots \\
\vdots & \vdots & \ddots & \vdots & \cdots \\
c_{n-1} & c_{n+1} & \cdots & c_{2n-2} & \cdots \\
\vdots & \vdots & \vdots & \vdots & \ddots
\end{pmatrix} &= H_\infty = V^TDV \\
&= \begin{pmatrix}
\sum_{i=1}^r b_i & \sum_{i=1}^r b_i a_i & \cdots & \sum_{i=1}^r b_i a_i^{n-1} & \cdots \\
\sum_{i=1}^r b_i a_i & \sum_{i=1}^r b_i a_i^2 & \cdots & \sum_{i=1}^r b_i a_i^{n} & \cdots \\
\vdots & \vdots & \ddots & \vdots & \cdots \\
\sum_{i=1}^r b_i a_i^{n-1} & \sum_{i=1}^r b_i a_i^n & \cdots & \sum_{i=1}^r b_i a_i^{2n-2} & \cdots \\
\vdots & \vdots & \vdots & \vdots & \ddots
\end{pmatrix}_.
\end{align*}
Thus, equating the entries of the two representations of $H_\infty$ yields $c_n = \sum_{i=1}^r b_i a_i^{n}$ for all $n\geq 0$. Using the transformation $\beta_i = a_i$ and $\frac{b_i}{\beta_i} = \alpha_i$ for $1\leq i\leq r$, we obtain the desired form. Note that this transformation is well-defined since $a_i\neq 0$ is a consequence of the definition of a non-degenerate Vandermonde decomposition. Moreover, all $a_i = \beta_i$ are distinct, since $a_i = a_j$ for $i\neq j$ implies the $i$th and $j$th rows of $V$ are equal, contradicting $\rank(V) = r$. 

Suppose (1) by letting $\mathcal{C}$ have moment $r$. Then there exists a set of nonzero complex numbers $\{\alpha_i\}_{i=1}^r$ and distinct nonzero $\{\beta_i\}_{i=1}^r$ so that $c_n=\sum_{i=1}^r \alpha_i \beta_i^{n+1}$ for all $n$. We see from the computation in Lemma \ref{Lem:VandeComp} that $D = \diag(\alpha_i\beta_i)$ and the $(i,j)$ entry of $V_{r\times\infty}$ given by $\beta_i^{j-1}$ is a non-degenerate Vandermonde decomposition of $H_\infty$. 

It only remains to show that $H_\infty$ has rank $r$. 
It is immediately clear that $\rank(H_\infty)\leq r$. Moreover, Frobenius's inequality \cite[0.4.5(e)]{HorJoh2013} implies that for the product $V^TDV$, 
$$
\rank(V^TD)+\rank(DV)\leq \rank(D)+\rank(V^TDV).
$$ 
We show now that $\rank(V^TD) = r$. Suppose that $\rank(V^TD)<r$. Then, not all columns of $V^TD$ are linearly independent, and there exists $\{\gamma_i\}_{i=1}^{r-1} \subset \mathbb{C}$ so that $\sum_{i=1}^{r-1} \gamma_i C_i = C_r$, where $C_i$ denotes the $i$th column of $V^TD$. By examining the entries within the columns of $V^TD$, we see that $\sum_{i=1}^{r-1} \gamma_i C_i = C_r$ implies $\sum_{i=1}^{r-1} \gamma_i\alpha_i\beta_i^{n+1} = \alpha_r\beta_r^{n+1}$ for all $n\geq 0$, leading to the following representation of $c_n$. 
$$
c_n = \sum_{i=1}^r \alpha_i \beta_i^{n+1} = \sum_{i=1}^{r-1} (1+\gamma_i)\alpha_i\beta_i^{n+1}
$$

This implies $\mathcal{C}$ has description complexity at most $r-1$, contradicting Lemma \ref{lem:GDCisMeaningful}. Therefore $\rank(V^TD) = r = \rank(DV)$, and Frobenius's inequality produces the desired result.  (It is also possible to argue this via Sylvester's Law of Inertia applied to the leading principal $r \times r$ submatrices of $H_\infty$, $V$, and $D$.)

\item[$1 \Leftrightarrow 8$:] It is clear that $c_n = \sum_{i=1}^r \alpha_i\beta_i^{n+1}$ is the $(n+1)$-st moment of the $r$-atomic complex measure $\mu = \sum_{i=1}^r \alpha_i \delta_{\beta_i}$, where $\delta_{\beta_i}$ denotes the standard Dirac delta function.

\item[$1 \Leftrightarrow 9$:] We argue that the algorithm returns the parameter $r$ if and only if $r$ is the smallest positive integer so that every term of the sequence $(c_n)$ can be expressed as $c_n = \sum_{i=1}^r \alpha_i \beta_i^{n+1}$ for a set of nonzero complex $\alpha_i$ and distinct $\beta_i$ (since $\alpha_i =0$ or $\beta_i = \beta_j$ for $i \neq j$ contradict the minimality of $r$). 

Suppose the algorithm succeeds.  Then $(c_n)_{n \geq 0}$ satisfies the recurrence $\sum_{n=0}^r a_n c_{n+t} = 0$, where the $\{a_n\}$ are defined by step 4; this is well-defined because step 3 says that $\det(\mathbf{C})\neq 0$.  The characteristic polynomial of the recurrence $p(x)$ is well-defined and factors into $r$ distinct linear factors $x-\beta_i$ by step 6.  Thus, by standard results in the theory of linear recurrence relations,
$$
c_n = \sum_{i=1}^r \alpha_i \beta_i^{n+1}
$$
for some $\{\alpha_i\}_{i=1}^r \subset \mathbb{C}$, which are nonzero by the minimality of $r$ (since $\malg(\mathcal{C})$ did not terminate on any $r' < r$).  So, $\mrank(\mathcal{C}) = r$.

Now, suppose that $c_n = \sum_{i=1}^r \alpha_i\beta_i^{n+1}$ with all distinct nonzero $\beta_i$ and nonzero values $\alpha_i$.  By $1 \Rightarrow 3$ above, the matrix $\mathbf{C} = H_{r-1,0}$ introduced in step 2 is invertible, so we pass step 3.  Write $\Phi(z) = \sum_{n=0}^\infty c_n z^n$.  Because $\max_i |\beta_i| < \infty$, the following is true in a sufficiently small ball about $z=0$:
\begin{align*}
\Phi(z) &= \sum_{n = 0}^\infty \sum_{i=1}^r z^n \alpha_i\beta_i^{n+1} \\
&= \sum_{i = 1}^r \frac{\alpha_i\beta_i}{1 - \beta_i z}.
\end{align*}
Thus, if we let $g(z) = \prod_{i=1}^r (1-\beta_i z) =: \sum_{i=0}^r a_{r-i} z^i$, then $\Phi(z)g(z)$ is a polynomial of degree $r-1$, i.e.,
\begin{align*}
\Phi(z)g(z) & = \sum_{n=0}^\infty c_n z^n g(z) = \sum_{n=0}^\infty c_n z^n \sum_{i=0}^r a_{r-i} z^i \\
& = \sum_{n=r}^\infty z^n \sum_{i=0}^r a_{r-i} c_{n-i} + \sum_{n=0}^{r-1} z^n \sum_{i=0}^n a_{r-i} c_{n-i} \\
& = \sum_{n=r}^\infty z^n \sum_{i=0}^r a_{i} c_{i+n-r} + \sum_{n=0}^{r-1} z^n \sum_{i=0}^n a_{i+r-n} c_{i}.
\end{align*}
By matching coefficients of $z^n$ on each side, we see, for $n \geq r$, $0 = \sum_{i=0}^r a_i c_{i+n-r}$, so we pass step 5.  But, $a_r = 1$, so setting $n = r$,
\begin{equation} \label{eq:thisrecurrence}
c_{r} = - \sum_{i=0}^{r-1} a_i c_{i},
\end{equation}
i.e.,
$$
\bC \left [ 
\begin{array}{c} 
a_0  \\
\vdots  \\
a_{r-1}  \\
\end{array} 
\right ] = - \mathbf{c}_r^T
$$
so
$$
\left [ 
\begin{array}{c} 
a_0  \\
\vdots  \\
a_{r-1}  \\
\end{array}  
\right ] = - \bC^{-1} \mathbf{c}_r^T
$$
Thus, $p(x) = \sum_{i=0}^{r} a_{i} x^{i}$ is the characteristic polynomial of the recurrence (\ref{eq:thisrecurrence}).  Since the recurrence is solved by $c_n = \sum_{i=1}^r \alpha_i \beta_i^{n+1}$, the roots of $p(x)$ are the distinct values $\{\beta_i\}_{i=1}^r$, and we pass step 6.  Since Lemma \ref{lem:GDCisMeaningful} implies that the algorithm does not succeed for any $r' < r$, the result follows.
\end{description}
\end{proof}

In condition (4), the quantity $r$ is known as the ``Waring rank'' of this polynomial.  Pratt \cite{Pra19} presents a history and many interesting results about this classical invariant.

Note that, when $\mrank{C}$ is finite, $\malg(\mathcal{C})$ returns $r$ (the order of the recurrence satisfied by $\mathcal{C}$) and a vector $\boldsymbol{\alpha} = (\mathbf{c}_0 \vdm(\boldsymbol{\beta})^{-1})^T$.  Since $c_n = \sum_{i=1}^r \alpha_i \beta_i^{n+1}$ for some nonzero values $\{\alpha_i\}_{i=1}^r$, we may write
$$
\vdm(\boldsymbol{\beta})^T \left [ 
\begin{array}{c} 
\alpha_1  \\
\alpha_2  \\
\vdots  \\
\alpha_r  \\
\end{array} 
\right ] = \left [ 
\begin{array}{cccc} 
\beta_1 & \beta_2 & \cdots & \beta_r \\
\beta_1^2 & \beta_2^2 & \cdots & \beta_r^2 \\
\vdots & \vdots & \ddots & \vdots \\
\beta_1^{r} & \beta_2^{r} & \cdots & \beta_r^{r} \\
\end{array} 
\right ] \left [ 
\begin{array}{c} 
\alpha_1  \\
\alpha_2  \\
\vdots  \\
\alpha_r  \\
\end{array} 
\right ] = \left [ 
\begin{array}{c} 
c_0  \\
c_1  \\
\vdots \\
c_{r-1} \\
\end{array} 
\right ] = \mathbf{c}_0^T.
$$
Since the $\{\beta_i\}$ are distinct when $r$ is minimal,
$\vdm(\boldsymbol{\beta})$ is invertible, so the vector of $\alpha_i$ values equals $(\vdm(\boldsymbol{\beta})^{-1})^T \mathbf{c_0}$, i.e., $\boldsymbol{\alpha} = (\alpha_1,\ldots,\alpha_r)$.

\begin{cor}
For a sequence $\mathcal{C} = (c_n)_{n=0}^\infty$ satisfying a simple $r$-th order linear recurrence for $r$ minimal, 
$$
\nul H_{m,t} = \max\{0,m-r+1\}
$$
\end{cor}
\begin{proof}
The rank of $H_{m,t}$ is equal to the order $r$ of the smallest linear recurrence satisfied by $\mathcal{C}$ for $m \geq r$ and is $m+1$ for $m < r$, since $H_{r,t}$ is invertible.  Therefore, $\nul(H_{m,t}) = \max\{0,m-r+1\}$
\end{proof}

\section{Unitary Rank}

Given a sequence $\cC = (c_n)_{n\geq 0}$ with moment rank $r$, we know that $c_n = \sum_{i=1}^r \alpha_i\beta_i^{n+1}$ for nonzero $\{\alpha_i\}$ and nonzero, distinct $\{\beta_i\}$. It will be useful for some applications to isolate and consider the case when $\alpha_i = 1$ for all $i$, meaning $c_n = \sum_{i=1}^r \beta_i^{n+1}$ for $n \geq 0$. We proceed with the following definition (changing the sequence index for ease of use later).

\begin{definition}\label{Def:urank}
The sequence $\mathcal{C} = (c_n)_{n=1}^N$ (with $N = \infty$ allowed) is said to have \textbf{unitary rank} $r$ if $r$ is the smallest positive integer so that there exists a multiset of nonzero complex values $\{\beta_i\}_{i=1}^r$ so that $c_n=\sum_{i=1}^r \beta_i^{n}$ for all $1 \leq n \leq N$. If $N = \infty$, we write $\urank((c_n)_{n \geq 1})$ for the unitary rank.
\end{definition}

\begin{lemma} \label{lem:DCisMeaningful}
$\urank((c_n)_{n \geq 1})$ is well-defined.
\end{lemma}
\begin{proof}
Suppose, by way of contradiction, that there are two distinct multi-sets $\{\beta_i\}_{i=1}^r$ and $\{\gamma_j\}_{j=1}^s$ of nonzero complex numbers so that  $$c_n=\sum_{i=1}^r\beta_i^{n}=\sum_{j=1}^s\gamma_j^{n}$$ Write $f(z)=\sum_{n=1}^\infty c_n z^n$. Because $\max_i|\beta_i|$ and $\max_j|\gamma_j|$ are finite, the following is true around a sufficiently small ball about $z=0$: 
\begin{align*} 
\sum_{n=1}^\infty \sum_{i=1}^r z^n\beta_i^{n} &= \sum_{n=1}^\infty \sum_{j=1}^s z^n\gamma_j^{n} \\
\sum_{i=1}^r \frac{z\beta_i}{1-z\beta_i} &= \sum_{j=1}^s \frac{z\gamma_j}{1-z\gamma_j}
\end{align*}

These two functions are equal, so they have the same set of poles; thus, as sets, $\{\gamma_i\}_i = \{\beta_i\}_i$.  Furthermore, since the residues of these poles are proportional to the multiplicity of the $\gamma_i$ and $\beta_i$ values, they also occur the same number of times.  Thus, as multisets as well, $\{\gamma_i\}_i = \{\beta_i\}_i$.

\end{proof}

\begin{theorem}\label{Thm:TFAEforDC}
Let $\mathcal{C} = (c_n)_{n=1}^\infty$ be a sequence and $r$ a positive integer. Let $H_{r-1,t}$ denote the $r \times r$ Hankel matrix whose entries come from the sequence $(c_n)_{n=t}^{2r-2+t}$. Lastly, let $\Phi(z) = \sum_{n \geq 1} c_n z^n$ denote the ordinary generating function of the sequence $\mathcal{C}$. Then the following are equivalent. 
\begin{enumerate}
\item The sequence $\mathcal{C}$ has unitary rank $r$. 
\item The algorithm $\malg(T[\mathcal{C}])$ succeeds on the sequence $T [\mathcal{C}] = (c_{n+1})_{n \geq 0}$ and returns $r'$ and $\boldsymbol{\alpha} \in \mathbb{N}^{r'}$ so that $\boldsymbol{\alpha} \cdot \mathbf{1} = r$, where $\mathbf{1}$ is the all-ones vector of dimension $r'$.  
\item $\mathcal{C}$ satisfies an $r'$-th order simple linear recurrence for some $r' \leq r$ with coefficients $\{\alpha_i\}_{i=1}^{r'}$, $\alpha_i \in\mathbb{N}$ for all $1\leq i\leq r'$, and $\sum_{i=1}^{r'}\alpha_i = r$. 
\item Let $\mathcal{C}' = (c'_n)_{n\geq 0}$ be a sequence so that $c'_0 = r$ and $c'_n = c_{n}$ for $n\geq 1$. If $H'_{r-1,t}$ is the $r \times r$ Hankel matrix whose first row consists of $c'_t,\ldots,c'_{t+r-1}$, then there exists an $r\times r$ Vandermonde matrix $V = \vdm(\boldsymbol{\beta})$ so that $H'_{r-1,t} = V^T D^t V$ where $D = \diag(\boldsymbol{\beta})$, for each $t\geq 0$. 
\item Let $\mathcal{C}' = (c'_n)_{n\geq 0}$ be a sequence so that $c'_0 = r$ and $c'_n = c_{n}$ for $n\geq 1$. If $H_\infty'$ is the infinite Hankel matrix with entries from $\mathcal{C}'$, then there exists a Vandermonde matrix $V_{r,\infty}'$ so that $H_\infty' = (V')^TV'$.
\item $\exp(\int -x^{-1} \Phi(x) \, dx)$ is a polynomial of degree $r$ with nonzero roots. 
\item There exist nonzero $r' \in \mathbb{N}$ and $\{\alpha_1,\ldots,\alpha_{r'}\}$, $\{\beta_1,\ldots,\beta_{r'}\}$, and $\{\lambda_1,\ldots,\lambda_{r'}\}$ so that, for each $t \geq 0$, the polynomial $\sum_{j=0}^{2r'} \binom{2r'}{j} c_{j+t} x^{2r'-j} y^{j} = \sum_{j=1}^{r'} \lambda_j (\beta_j/\alpha_j)^t (\alpha_j x + \beta_j y)^{2r'}$,  
and the $\{\alpha_j\}_{j=1}^{r'}$ are positive naturals that sum to $r$. 

\item The sequence $\mathcal{C}$ is the moment sequence for a complex finite-atomic measure on $\mathbb{C}$ where the masses of the atoms are positive naturals with sum $r$. 

\end{enumerate}
\end{theorem}
\begin{proof}

\begin{description}

\item[$1 \Leftrightarrow 3$:] This is clear from the definitions of $\urank$ and $\mrank$, taking into account the shift in index. 

\item[$1\Leftrightarrow 2$:] The sequence $\mathcal{C}$ having unitary rank $r$ implies there exists $r'\leq r$ so that $\mrank(T[\mathcal{C}]) = r'$, since we may take the same $\beta_i$ with $\alpha_i$ equal to the multiplicity of $\beta_i$ in the representation $c_n = \sum_{i=1}^r \beta_i^n$.  Thus, the algorithm applied to $T[\mathcal{C}]$ with Hankel matrix of size $r'$ succeeds and returns $r'$ and $\boldsymbol{\alpha} = (\mathbf{c}_0 \vdm(\boldsymbol{\beta})^{-1})^T$, which, by the note following the proof of Theorem \ref{Thm:TFAEforGDC}, is the vector of coefficients $\alpha_i$ in the representation $c_n = \sum_{i=1}^{r'} \alpha_i \beta_i^{n}$.  (Note that the exponent of $\beta_i$ is $n$ instead of $n+1$ because $T[\mathcal{C}]$ is the input to $\malg$ here.)  But then, the sum $\sum_{i=1}^{r'} \alpha_i$ over the values generated by the algorithm is $r$, since this is the number of terms in the representation $c_n = \sum_{i=1}^r \beta_i^n$.

To establish the converse direction, suppose the algorithm returns $r'$ and $\boldsymbol{\alpha} \in \mathbb{N}^{r'}$ with $\boldsymbol{\alpha} \cdot \mathbf{1} = r$.  Note that none of the $\alpha_i$ are zero, since then $\malg$ would have terminated on a smaller value of $r'$.  Thus, the $\alpha_i$ are positive integers summing to $r$, so we may write
$$
c_{n+1} = \sum_{j=1}^{r'} \sum_{i=1}^{\alpha_i} \beta_j^{n+1},
$$
i.e., $c_n = \sum_{j=1}^r {\beta'_j}^{n}$ if $\{\beta_j\}_{j=1}^{r'} = \{\beta'_j\}_{j=1}^r$ with $\beta_j$ appearing with multiplicity $\alpha_j$ on the right-hand side.

\item[$3\Leftrightarrow 8$:] This is a direct consequence of Theorem \ref{Thm:TFAEforGDC}.

\item[$1 \Leftrightarrow 5$:] Suppose condition $(1)$ holds and let $\{\beta_i\}_{i=1}^r$ be given to satisfy the definition of unitary rank for the sequence $\mathcal{C}$. To adopt the notation of Lemma \ref{Lem:VandeComp}, letting $b_i = 1$ and $a_i = \beta_i$ for $1\leq i\leq r$ proves $(5)$. 

Now, suppose $(5)$ holds. Letting $b_i = 1$ for each $1\leq i\leq r$ in the form of $V^TDV$ given in Lemma \ref{Lem:VandeComp}, we have that the $(i,j)$ entry of $H_\infty'$ is $\sum_{i=1}^r a_i^{i+j-2}$. Clearly then, we have that $c_n' = \sum_{i=1}^r a_i^n$, further implying that $c_n = \sum_{i=1}^r a_i^{n+1}$. The uniqueness of unitary rank given by Lemma \ref{lem:DCisMeaningful} completes the proof. 

\item[$1\Leftrightarrow 6$:] Suppose $(1)$ holds. Let $c_n = \sum_{i=1}^r \beta_i^{n}$. Since $f(x) = \sum_{n\geq 1} c_nx^n$, we have that $\int -x^{-1} f(x) \, dx = -\sum_{n\geq 1} \frac{c_n x^{n}}{n} + C$. We have the following computation.  
\begin{align*}
\exp \left( -\sum_{n\geq 1} \frac{c_n x^{n}}{n} \right) &= \exp \left( - \sum_{n\geq 1} \sum_{i=1}^r \frac{\beta_i^{n} x^{n}}{n} \right) \\
&= \exp \left( - \sum_{i=1}^r \sum_{n\geq 1} \frac{(\beta_ix)^{n}}{n} \right) \\
&= \prod_{i=1}^r \exp \left( - \sum_{n\geq 1} \frac{(\beta_ix)^n}{n} \right) \\
&= \prod_{i=1}^r \exp \left( \log[1-\beta_i x] \right) \\
&= \prod_{i=1}^r (1-\beta_i x)
\end{align*}
Therefore, we see that $\int -x^{-1} f(x) \, dx$ is the log of a polynomial of degree $r$, as desired. 

Suppose $(6)$ holds. Since it is only possible to take the log of a polynomial if the polynomial has nonzero constant term, we have that there exists nonzero $\{\beta_i\}_{i=1}^r$ so that $\exp(-\int x^{-1} f(x) \, dx) = \prod_{i=1}^r (1-\beta_i x)$. By letting $c_n = \sum_{i=1}^r \beta_i$ and working backwards through the computation given in the first direction of the proof shows 
$$
\exp \left( -\int x^{-1} f(x) \, dx \right) = e^C \exp \left(- \sum_{n\geq 1} \frac{c_n x^{n}}{n} \right)
$$
Taking log of both sides and differentiating gives that $f(x) = \sum_{n\geq 1} c_n x^{n}$, showing that $f(x)$ is the generating function for a sequence with unitary rank $r$, completing the proof. 

\item[$1 \Leftrightarrow 7$:] Suppose $(1)$ holds. Having already proved $1\Leftrightarrow 3$, we appeal to $(3)$ and Theorem \ref{Thm:TFAEforGDC}, to give that (despite the shift in index, which is taken care of by the variable $t$) there exist nonzero $\{\alpha_1,\ldots,\alpha_{r'}\}$, $\{\beta_1,\ldots,\beta_{r'}\}$, and $\{\lambda_1,\ldots,\lambda_{r'}\}$ so that, for each $t \geq 0$, the polynomial $\sum_{j=0}^{2r'} \binom{2r'}{j} c_{j+t} x^{2r'-j} y^{j} = \sum_{j=1}^{r'} \lambda_j (\beta_j/\alpha_j)^t (\alpha_j x + \beta_j y)^{2r'}$. The fact that $\sum_{i=1}^{r'} \alpha_i = r$ is given by the assumption of condition $(1)$. 

Now suppose $(7)$ is true. By Theorem \ref{Thm:TFAEforGDC}, we have that $\mrank(\mathcal{C}) = r'$. The assumption in $(7)$ that $\{\alpha_i\}_{i=1}^{r'}$ is a set of positive naturals summing to $r$ gives that $\urank(\mathcal{C}) = r$. 

\item[$1 \Leftrightarrow 4$:] Suppose $(4)$. By the computation given in Lemma \ref{Lem:VandeComp}, it is clear that the $(i,j)$ entry of $V^T D^t V$ is given by $\sum_{i=1}^r \beta_i^{t+i+j-2}$. By letting $c_n = \sum_{i=1}^r \beta_i^n$, we have that $\mathcal{C} = (c_n)_{n=1}^\infty$ has unitary rank at most $r$. We note that Lemma \ref{lem:DCisMeaningful} completes the proof that $\urank(\mathcal{C}) = r$. 

Now, suppose $(1)$. By the computation given in Lemma \ref{Lem:VandeComp}, we see that letting the $i$th row of $V$ be generated by $\beta_i$ and defining $D = \diag(\beta_i)$, the result follows immediately. 

\end{description}

\end{proof}

Suppose $A$ is an infinite positive-semidefinite matrix (aka positive-type kernel for $\ell^2$); then $A$ can be written as $A = M^\ast M$, which we will refer to as a ``Gramian representation'' (since $M^\ast M$ is a Gramian matrix in the finite-dimensional case). In general, this representation is unique up to unitary conjugation. 

\begin{cor}
The property that a real sequence $\mathcal{C}$ is the moment sequence of a finite-atomic measure on $\mathbb{C}$ with integer masses is equivalent to $H'_\infty$ being positive semi-definite of finite rank with a Vandermonde Gramian representation.
\end{cor}
\begin{proof}
By Theorem 3, $H'_\infty$ has a factorization as $V^T V$ for some Vandermonde kernel $V$.  Note that, if $H'_\infty = M^\ast M$ for some $M$, then $H'_\infty$ is Hermitian and symmetric, which implies that $\mathcal{C}$ is real.  Thus, $V^T = V^\ast$.
\end{proof}

\bibliography{ref}

\begin{thebibliography}{10}

\bibitem{Ble13}
Grigoriy Blekherman.
\newblock Nonnegative polynomials and sums of squares.
\newblock In {\em Semidefinite optimization and convex algebraic geometry},
  volume~13 of {\em MOS-SIAM Ser. Optim.}, pages 159--202. SIAM, Philadelphia,
  PA, 2013.

\bibitem{BolLukVan}
Daniel~L. Boley, Franklin~T. Luk, and David Vandevoorde.
\newblock Vandermonde factorization of a {H}ankel matrix.
\newblock In {\em Scientific computing ({H}ong {K}ong, 1997)}, pages 27--39.
  Springer, Singapore, 1997.

\bibitem{BosElvGutMai2020}
Alin Bostan, Andrew Elvey-Price, Anthony Guttmann, and Jean-Marie Maillard.
\newblock Stieltjes moment sequences for pattern-avoiding permutations.
\newblock {\em The Electronic Journal of Combinatorics}, to appear.

\bibitem{Bre89}
Francesco Brenti.
\newblock Unimodal, log-concave and {P}\'{o}lya frequency sequences in
  combinatorics.
\newblock {\em Mem. Amer. Math. Soc.}, 81(413):viii+106, 1989.

\bibitem{ChuLin15}
Moody~T. Chu and Matthew~M. Lin.
\newblock On the finite rank and finite-dimensional representation of bounded
  semi-infinite {H}ankel operators.
\newblock {\em IMA J. Numer. Anal.}, 35(3):1256--1276, 2015.

\bibitem{CoSi57}
Lothar Collatz and Ulrich Sinogowitz.
\newblock Spektren endlicher {G}rafen.
\newblock {\em Abh. Math. Sem. Univ. Hamburg}, 21:63--77, 1957.

\bibitem{CurFiaMol2008}
Raúl~E. Curto, Lawrence~A. Fialkow, and H.~Michael Möller.
\newblock The extremal truncated moment problem.
\newblock {\em Integral Equations and Operator Theory}, 60:177--200, 2 2008.

\bibitem{DioSch2017}
Philipp~J. di~Dio and Konrad Schm\"{u}dgen.
\newblock The multidimensional truncated moment problem: atoms, determinacy,
  and core variety.
\newblock {\em J. Funct. Anal.}, 274(11):3124--3148, 2018.

\bibitem{EvePooShpWar2003}
Graham Everest, Alf van~der Poorten, Igor Shparlinski, and Thomas Ward.
\newblock {\em Recurrence sequences}, volume 104 of {\em Mathematical Surveys
  and Monographs}.
\newblock American Mathematical Society, Providence, RI, 2003.

\bibitem{FomZel2000}
Sergey Fomin and Andrei Zelevinsky.
\newblock Total positivity: Tests and parametrizations.
\newblock {\em The Mathematical Intelligencer}, 22:23--33, 2000.

\bibitem{Fre07}
Christopher French.
\newblock Transformations preserving the {H}ankel transform.
\newblock {\em J. Integer Seq.}, 10(7):Article 07.7.3, 14, 2007.

\bibitem{Ha23}
Felix Hausdorff.
\newblock Momentprobleme f\"{u}r ein endliches {I}ntervall.
\newblock {\em Math. Z.}, 16(1):220--248, 1923.

\bibitem{HorJoh2013}
Roger~A. Horn and Charles~R. Johnson.
\newblock {\em Matrix analysis}.
\newblock Cambridge University Press, Cambridge, second edition, 2013.

\bibitem{Jun2003}
Alexandre Junod.
\newblock Hankel determinants and orthogonal polynomials.
\newblock {\em Expo. Math.}, 21(1):63--74, 2003.

\bibitem{Lay2001}
John~W. Layman.
\newblock The {H}ankel transform and some of its properties.
\newblock {\em J. Integer Seq.}, 4(1):Article 01.1.5, 11, 2001.

\bibitem{PeaWoa00}
Paul Peart and Wen-Jin Woan.
\newblock Generating functions via {H}ankel and {S}tieltjes matrices.
\newblock {\em J. Integer Seq.}, 3(2):Article 00.2.1, 1 HTML document, 2000.

\bibitem{Pin2010}
Allan Pinkus.
\newblock {\em Totally Positive Matrices}.
\newblock Cambridge University Press, 2010.

\bibitem{Pra19}
Kevin Pratt.
\newblock Waring rank, parameterized and exact algorithms.
\newblock In {\em 2019 IEEE 60th Annual Symposium on Foundations of Computer
  Science (FOCS)}, pages 806--823, Los Alamitos, CA, USA, nov 2019. IEEE
  Computer Society.

\bibitem{Rez2013}
Bruce Reznick.
\newblock On the length of binary forms.
\newblock In {\em Quadratic and higher degree forms}, volume~31 of {\em Dev.
  Math.}, pages 207--232. Springer, New York, 2013.

\bibitem{Sch2017}
Konrad Schmüdgen.
\newblock {\em The Moment Problem}.
\newblock Springer, 2017.

\bibitem{ShoTam1943}
J.~Shohat and J.~Tamarkin.
\newblock {\em The Problem of Moments}, volume~1.
\newblock American Mathematical Society, 12 1943.

\bibitem{Stu08}
Bernd Sturmfels.
\newblock {\em Algorithms in invariant theory}.
\newblock Texts and Monographs in Symbolic Computation. Springer, second
  edition, 2008.

\bibitem{Syl1851}
J.J. Sylvester.
\newblock {LX.}~{O}n a remarkable discovery in the theory of canonical forms
  and of hyperdeterminants.
\newblock {\em The London, Edinburgh, and Dublin Philosophical Magazine and
  Journal of Science}, 2(12):391--410, 1851.

\bibitem{Wid1966}
Harold Widom.
\newblock Hankel matrices.
\newblock {\em Transactions of the American Mathematical Society}, 121:1--35, 1
  1966.

\end{thebibliography}
\bibliographystyle{plain}

\end{document}